\documentclass[11pt]{amsart}
\usepackage{amssymb,amsthm}
\usepackage{hyperref}
\allowdisplaybreaks

\newcommand{\R}{{\mathbb{R}}}

\newcommand{\Z}{{\mathbb{Z}}}
\newcommand{\op}{\textit{op}}
\newcommand{\NN}{{2^{\mathbb N}}}

\newcommand{\qtq}[1]{\quad\text{#1}\quad}

\newcommand{\sep}{\par\vspace*{2ex}\hbox to \textwidth{\hss---\hss}\vspace{2ex}\par}
\newcommand{\ddt}{\frac{d\ }{dt}}
\DeclareMathOperator{\tr}{tr}
\DeclareMathOperator{\dist}{dist}
\DeclareMathOperator{\signum}{signum}
\let\Re=\undefined\DeclareMathOperator{\Re}{Re}

\newtheorem{theorem}{Theorem}[section]
\newtheorem{prop}[theorem]{Proposition}
\newtheorem{lemma}[theorem]{Lemma}
\newtheorem{corollary}[theorem]{Corollary}
\theoremstyle{definition}
\newtheorem{definition}[theorem]{Definition}
\theoremstyle{remark}
\newtheorem*{remark}{Remark}

\begin{document}

\title[Low regularity conservation laws for integrable PDE]{Low regularity conservation laws\\for integrable PDE}

\author{Rowan Killip, Monica Vi\c{s}an, and Xiaoyi Zhang}

\address
{Rowan Killip\\
Department of Mathematics,\\
University of California, Los Angeles, CA 90095, USA}
\email{killip@math.ucla.edu}

\address
{Monica Vi\c{s}an\\
Department of Mathematics,\\
University of California, Los Angeles, CA 90095, USA}
\email{visan@math.ucla.edu}

\address
{Xiaoyi Zhang\\
Department of Mathematics,\\
University of Iowa, Iowa City, IA 52242, USA}
\email{xiaoyi-zhang@uiowa.edu}

\begin{abstract}
We present a general method for obtaining conservation laws for integrable PDE at negative regularity and exhibit its application to KdV, NLS, and mKdV.  Our method works uniformly for these problems posed both on the line and on the circle.
\end{abstract}

\maketitle

\section{Introduction}

The original goal of this work was to obtain low-regularity conservation laws for the Korteweg--de Vries equation
\begin{align}\tag{KdV}\label{KdV}
\ddt q = - q''' + 6qq' .
\end{align}
However, the method we developed to address this problem turns out to be of more general validity, as we shall demonstrate by applying it to the cubic NLS and mKdV equations.  All three equations can be posed both on the real line $\R$ and on the circle $\R/\Z$.  (The latter case is equivalent to that of spatially periodic initial data on the whole line.)  The methods we use here apply equally well in both settings and correspondingly, we shall be treating them in parallel.

Naturally, the existence of conservation laws must be predicated on the existence of solutions.  Conversely, it is difficult to construct solutions without control on the growth of crucial norms.  As is usual in the study of PDE, our approach here is to only consider solutions that are smooth and rapidly decaying, but prove results that are uniform in low-regularity norms.  The fact that Schwartz-space initial data lead to unique global solutions to \eqref{KdV} that remain in Schwartz class has been known for some time; see, for example, \cite{MR0385355,MR0759907,Sj,MR0410135,MR0261183}.  (Recall that Schwartz space on $\R/\Z$ is coincident with $C^\infty(\R/\Z)$; on the line, it is comprised of those $C^\infty(\R)$ functions that decay faster than any polynomial as $|x|\to\infty$.)

It has been known since \cite{MR0252826} that \eqref{KdV} admits infinitely many conservation laws. The first three are
$$
\int q(t,x)\,dx, \qquad \int q(t,x)^2 \,dx, \qtq{and} \int \tfrac12 q'(t,x)^2 + q(t,x)^3\,dx.
$$ 
We can regard this original family of conservation laws as ordered: Each conserved quantity is a polynomial in $q$ and its derivatives that is scaling homogeneous.  Thus, they can be ordered by scaling, or equivalently, by the highest order derivative that appears.  The exact form of these conservation laws is rather delicate; in particular, they need not be sign-definite.  Nonetheless, they can be used (cf. \cite[\S3]{MR0369963}) to show that for an integer $s\geq  0$, the $H^s$-norm of the solution admits a global in time bound depending only on the corresponding norm of the initial data.  For well-posedness questions, such global bounds are of greater significance than the particular conservation laws that begot them; correspondingly, our presentation will emphasize such bounds, beginning with the following result.

\begin{theorem}\label{T:main} Fix $-1\leq s<1$ and let $q$ be a Schwartz solution to \eqref{KdV} either on $\R$ or on $\R/\Z$.  Then
\begin{align*}
\| q(0) \|_{H^s} \bigl( 1 +  \| q(0) \|_{H^s}\bigr)^{-\frac23} \lesssim \| q(t) \|_{H^s} \lesssim \| q(0) \|_{H^s} \bigl( 1 +  \| q(0) \|_{H^s}^2\bigr).
\end{align*}
\end{theorem}

A number of instances of this result have appeared before.  In the torus case, this result is completely subsumed by \cite{MR2179653}.  On the real line, the case $s=-1$ was treated in \cite{MR3400442}.  There is also the work \cite{KT}, contemporaneous with our own, which covers the full range $s\geq -1$ in the line case.  We claim two particular merits for our method here: (i) it is much simpler, and (ii) it works uniformly on both the line and the circle.  Our methods also allow us to obtain a priori bounds in Besov spaces (see Section~\ref{S:3} for the definition).  Specifically, we shall prove the following:

\begin{theorem}\label{T:Bmain} Fix parameters $r$ and $s$ conforming to either of the following restrictions:
$s=-1$ and $1\leq r\leq 2$, or $-1< s < 1$ and $1\leq r\leq \infty$.  If $q$ is a Schwartz solution to \eqref{KdV} either on $\R$ or on $\R/\Z$, then
\begin{align*}
\| q(0) \|_{B^{s,2}_r} \bigl( 1 +  \| q(0) \|_{B^{s,2}_r}^2\bigr)^{-\frac23} \lesssim \| q(t) \|_{B^{s,2}_r} \lesssim \| q(0) \|_{B^{s,2}_r} \bigl( 1 +  \| q(0) \|_{B^{s,2}_r}^2\bigr).
\end{align*}
\end{theorem}

As a rather obvious corollary of this result, we see that local well-posedness in any of these Besov spaces can be immediately upgraded to global well-posedness.  We note two particular applications of this:  First, our result provides a simpler alternative to \cite{MR1969209}, which extended the local well-posedness results of \cite{MR2054622,MR2531556,MR1329387,MR2501679} to global well-posedness.  The range of Sobolev spaces so obtained ($s\geq -\frac12$ on the circle and $s\geq -\frac34$ on the line) were shown in \cite{MR2018661} to be sharp for analytic well-posedness; that is, the data-to-solution map cannot be analytic on any larger $H^s$ space.

The bounds shown in \cite{MR1969209} are not uniform in time; indeed, to transfer local well-posedness to global well-posedness, one need only show that the norm does not blow up in finite time.  We also wish to draw attention to the paper \cite{MR3292346}, which proves a priori bounds in $H^{-4/5}(\R)$ locally in time.
 
As a second application, we note that Theorem~\ref{T:Bmain}  extends to global-in-time the analytic local well-posedness result of Koch \cite[Theorem~6.6]{Koch} in the $B^{-3/4,2}_\infty(\R)$ norm, which yields the lowest regularity at which global well-posedness on the line is known at this time (irrespective of the degree of regularity of the data-to-solution map).  On the circle, however, \cite{MR2267286} shows that \eqref{KdV} is globally well-posed in $H^s$ for all $s\geq -1$.  Conversely, \cite{MR2830706} shows there is no hope of even local well-posedness for $s<-1$ in either geometry.

Let us turn now to a brief overview of our method.  The Lax pair formulation of \eqref{KdV} provides an elegant expression of the complete integrability of this equation, that transcends the distinction between the problem on the line with decaying and periodic data.  We recall the following from \cite{MR0235310}:  If
\begin{align*}
L(t) &:= -\partial_x^2 + q(t,x), \\
P(t) &:= - 4 \partial_x^3 + 3\bigl(\partial_x q(t,x) + q(t,x) \partial_x \bigr),
\end{align*}
then
\begin{align}
\text{$q(t)$ solves \eqref{KdV}} \iff \ddt L(t) = [P(t),\, L(t)].        \label{Lax}
\end{align}
In particular, if $q(t)$ is a Schwartz-space solution to \eqref{KdV} then the unitary operators $U(t)$ defined via
$$
\ddt U(t) = P(t) U(t) \qtq{with} U(0) = \textrm{Id}
$$
conjugate $L(0)$ with $L(t)$; specifically,
\begin{align}\label{unitary}
L(t) = U(t) L(0) U(t)^*.
\end{align}
Thus we may say that the KdV flow ``preserves all spectral properties of $L$''.

It is at this point that the study of the two geometries typically diverges.  In the periodic case, the spectrum (as a set) already carries an enormous amount of information, usually expressed in terms of the gap lengths.  (Even gaps of length zero can be located in the spectrum by examining harmonic capacity.)  In the decaying case, however, the spectrum (as a set) carries relatively little information.  While the location of any bound states is non-trivial information, the essential spectrum always simply fills $[0,\infty)$.  Beginning already with \cite{GGKM}, the remedy has been to consider scattering data.

Let us regress for a moment:  A symmetric matrix is uniquely determined up to unitary equivalence by its characteristic polynomial.  There is no hope of constructing a characteristic polynomial for $L(t)$ directly --- it is an unbounded operator.   Therefore, we must renormalize.  This line of thinking leads quickly to a quantity known as the perturbation determinant,
\begin{align}\label{formal pert det}
 \det\left( \bigl(-\partial_x^2 + q + \kappa^2\bigr)\Big/\bigl(-\partial_x^2 + \kappa^2\bigr) \right),
\end{align}
which formally represents the ratio of the `characteristic polynomials' of $L(t)$ and that of the Schr\"odinger operator with no potential.

To proceed further, we introduce the notation
$$
R_0:=(-\partial_x^2 + \kappa^2 )^{-1}
$$
for the resolvent of the Schr\"odinger operator with zero potential, acting on $L^2(\R)$ in the line setting or on $L^2(\R/\Z)$ when considering the circle case.  In either case, this is well-defined whenever $\Re \kappa >0$; however, for simplicity it will suffice to consider $\kappa>0$ in all that follows.  With this notation in place and taking $q$ to be a Schwartz function, the naive expression \eqref{formal pert det} can be rewritten as
$$
\det\bigl(1 + q R_0\bigr),
$$
which is well defined for all $\Re \kappa >0$, both in the sense of Fredholm \cite{MR1554993} and in the more modern Hilbert-space sense (cf. \cite{MR2154153}).  Moreover, one obtains the same function of $\kappa$ if one renormalizes on the right, rather than on the left:
$$
\det\bigl(1 + R_0 q\bigr) = \det\bigl(1 + q R_0\bigr).
$$

In order to pursue our goal of studying the low regularity problem, we must introduce two further renormalizations.  The first problem is one of singularities: the operator $qR_0$ is not bounded unless $q$ is locally square integrable.  Our remedy is to put `half' of the free resolvent on either side and so consider
$$
\det\bigl(1 + \sqrt{R_0} q \sqrt{R_0} \bigr) .
$$
As we take $\kappa>0$, $R_0$ is a positive semi-definite operator; by the square-root of $R_0$, we mean the positive semi-definite square-root.

The second problem (appearing solely in the line case) is one of decay.  If, for example, we take $q\leq 0$, then all terms in the Fredholm expansion of the determinant are negative and the leading term is
\begin{align}\label{E:trace term}
 \tr\bigl( \sqrt{R_0} q \sqrt{R_0} \bigr) = \tfrac{1}{2\kappa} \int q(x)\,dx.
\end{align}
Evidently, this will not extend to functions $q$ that are merely $L^2$ decaying.  Our remedy here is to employ a device of Hilbert \cite{Hilbert} and work with
$$
\det{}_2 \bigl(1 + \sqrt{R_0} q \sqrt{R_0} \bigr) := \det\bigl(1 + \sqrt{R_0} q \sqrt{R_0} \bigr) \exp\bigl\{-\tr\bigl(\sqrt{R_0} q \sqrt{R_0}\,\bigr)\bigr\}.
$$

As noticed already by Hilbert, this renormalization allows one to extend the notion of determinant to the class of Hilbert--Schmidt operators (cf. \eqref{HS norm}).  This is good news: a simple computation (cf. Proposition~\ref{P:free}) reveals that $\sqrt{R_0} q \sqrt{R_0}$ is Hilbert--Schmidt if and only if $q\in H^{-1}$.  In fact (and this will be important), the Hilbert--Schmidt norm of this operator is comparable to the $H^{-1}$ norm of the function $q$.

With one more transformation, which is solely for cosmetic/expository purposes, we finally arrive at our central object:
\begin{align}\label{KdValpha fake}
\alpha(\kappa;q) = -\log \det{}_2 \bigl(1 + \sqrt{R_0} q \sqrt{R_0} \bigr).
\end{align}
In fact (and this is one of the reasons for taking a logarithm), we may avoid the theory of operator determinants by simply expanding in a series:
\begin{align}\label{KdValpha}
\alpha(\kappa;q) = \sum_{\ell=2}^\infty \frac{(-1)^\ell}{\ell} \tr\Bigl\{ \bigl( \sqrt{R_0}\, q\, \sqrt{R_0}\bigr)^\ell \Bigr\}.
\end{align}
Later we will see that this series converges provided $\kappa$ is sufficiently large, because this makes the Hilbert--Schmidt norm of $\sqrt{R_0}\, q\, \sqrt{R_0}$ small.  But in this regime, we may reasonably conflate this series with its first term, which by Proposition~\ref{P:free} is comparable to $\| q \|_{H^{-1}}^2$.   In this way, conservation of $\alpha(\kappa;q(t))$ under the KdV flow guarantees global control on the $H^{-1}$ norm of the solution.  This is precisely how we will prove Theorem~\ref{T:main} in the case $s=-1$; see Section~\ref{S:2} for details.

Although we have led the reader to expect that $\alpha(\kappa,q(t))$ is conserved under the KdV flow, this does not seem to follow elementarily from \eqref{Lax} or \eqref{unitary}.   Indeed, an example discussed in Section~\ref{S:2} provides an obstruction that any such abstract argument would need to circumvent.  Instead, we give a simple direct proof that $\alpha(\kappa,q(t))$ is conserved, based on matching terms in the derivative of the expansion \eqref{KdValpha}.

An alternate approach to showing that the (renormalized) perturbation determinant is conserved would be to connect it to the conserved quantities used more traditionally, namely, the transmission coefficient (for decaying data) and the discriminant (for periodic data).  The arguments of \cite{MR0044404}, for example, show that the (unrenormalized) perturbation determinant coincides with the reciprocal of the transmission coefficient; however, we contend that these arguments (and indeed even the definition of the transmission coefficient) are rather more complicated than what we present here.

As we outlined above, the $s=-1$ case of Theorem~\ref{T:main} follows by considering $\alpha(\kappa;q(t))$ at a single (sufficiently large) value of $\kappa$.  Theorem~\ref{T:Bmain} and the remaining cases of Theorem~\ref{T:main} rely on combining several values of $\kappa$.  The simplest example of this is Theorem~\ref{T:main} in the regime $-1<s<0$, which follows from our argument at $s=-1$ simply though the integral \eqref{secIntegral}.  (See also the proof of Corollary~\ref{C:neg} for the closely related argument applicable to Besov norms for the same range of $s$.)  The heuristic principle at play here is that $\alpha(\kappa)$ captures the $L^2$ norm of the part of $q$ living at frequencies $|\xi| \lesssim \kappa$.  It is then simply a matter of expressing the relevant Sobolev or Besov norms as combinations of such quantities.

The argument just sketched cannot extend directly to $s>0$, because in this regime high-frequencies are more heavily weighed than low frequencies.  To overcome this obstacle, we consider a properly weighted difference of $\alpha(\kappa)$ and $\alpha(2\kappa)$ that (heuristically at least) captures the $L^2$ norm of the part of $q$ living at frequencies $|\xi|\sim \kappa$.  These can then be combined to yield conservation laws comparable to the Sobolev and Besov norms of interest.

One small hiccup associated to the passage to $s>0$ is that it no longer suffices to dismiss the tail of the series \eqref{KdValpha} as being merely a small percentage of the leading term.  Taking differences, requires us to estimate it in absolute value.

The limitation to $s<1$ stems from a breakdown of the crude heuristic outlined above: frequencies outside the regime $|\xi|\sim \kappa$ do contribute to the difference and this contribution becomes unacceptable once $s\geq 1$.  To pass to higher values of $s$, one must use more sophisticated differencing involving more sample points and the simplicity of our method begins to erode.  We have no plans to pursue this direction.

The NLS and mKdV models mentioned at the beginning of this paper actually fit under a common umbrella: they admit a Lax pair with the same operator $L(t)$ (but different operators $P(t)$).  The fact that this operator acts on vector-valued functions is of no consequence to our method.  There are, however, two meaningful changes that we must discuss.

The first is a simplification: the leading term in the series defining the perturbation determinant is already quadratic in $q$ and so there is no need to use renormalized determinants.

The second change is that the series \eqref{AKNS alpha} involves a non-selfadjoint operator and concomitant with this, the leading term no longer dominates the series in any self-evident way.  In this regard, these problems are more closely analogous to the discussion of KdV for $s>0$ appearing in Section~\ref{S:3}, rather than the simpler treatment of $-1\leq s<0$ in Section~\ref{S:2}.

Our principal result for NLS and mKdV on Sobolev spaces $H^{s}$ is the following:

\begin{theorem}\label{T:AKNS Hs}
Fix $-\tfrac12 < s <0$ and let $q(t)$ be a Schwartz solution to \eqref{NLS} or \eqref{mKdV}.  Then
\begin{align*}
\| q(t) \|_{H^s} \lesssim  \| q(0) \|_{H^s} \Bigl( 1 + \| q(0) \|_{H^s}^2 \smash{ \Bigr)^{\frac{|s|}{1-2|s|}} }.
\end{align*}
\end{theorem}

We also obtain an analogous result in  Besov spaces; see Theorem~\ref{T:AKNS Besov}.    Independently of us, Koch and Tataru \cite{KT} proved an analogue of Theorem~\ref{T:AKNS Hs} for all $s>-\frac12$ in the line case, by a method diverging sharply from our own.   

 The restriction to $s> -\frac12$ is necessary here, because solutions to the cubic NLS can undergo arbitrarily large inflation of the $H^{-1/2}$ norm starting from arbitrarily small data; see \cite{CarlesKappeler,Kishimoto,Oh:inflate}.   Tempted by the fact that this endpoint is forbidden, we have investigated how closely we might approach it by introducing logarithmic terms into the definition of Besov and Sobolev norms.  The results of these investigations are given in Theorem~\ref{T:XY}.

The paper is organised as follows:  The remainder of the introduction is devoted to the introduction of our preferred notations and some background on operator traces. In Section~\ref{S:2} we prove Theorem~\ref{T:main} for $-1\leq s <0$.  As we have explained, this can be done in a very simple self-contained way.

Section~\ref{S:3} begins with the introduction of Besov spaces and then proceeds to the proof of Theorem~\ref{T:Bmain} along the lines laid out above.  The treatment of mKdV and NLS comprises Section~\ref{S:4}.

The arguments of Sections~\ref{S:2} and~\ref{S:3} were worked out during the authors' stay at MSRI in the Fall of 2015 and were first presented at that time, \cite{K:MSRI}.  The arguments of Section~\ref{S:4} were first presented in \cite{K:Bonn}, albeit only to prove Theorem~\ref{T:AKNS Hs} above.

\subsection*{Acknowedgements}  This material is based on work supported by the National Science Foundation under Grant No. 0932078000 while the authors were in residence at the Mathematical Sciences Research Institute in Berkeley, California, during the Fall 2015 semester.  R. K. was additionally supported by NSF grant DMS-1600942 and M. V. by NSF grant DMS-1500707.

\subsection{Notation}
Our conventions for the Fourier transform are as follows:
\begin{align*}
\hat f(\xi) = \tfrac{1}{\sqrt{2\pi}} \int_\R e^{-i\xi x} f(x)\,dx  \qtq{so} f(x) = \tfrac{1}{\sqrt{2\pi}} \int_\R e^{i\xi x} \hat f(\xi)\,d\xi
\end{align*}
for functions on the line and
\begin{align*}
\hat f(\xi) = \int_0^1 e^{- i\xi x} f(x)\,dx \qtq{so} f(x) = \sum_{\xi\in 2\pi\Z} \hat f(\xi) e^{i\xi x}
\end{align*}
for functions on the circle $\R/\Z$.  Concomitant with this, we define
\begin{align*}
\| f\|_{H^{s}(\R)}^2 = \int_\R |\hat f(\xi)|^2 (1+|\xi|^2)^s \,d\xi
\end{align*}
and
\begin{align*}
\| f\|_{H^{s}(\R/\Z)}^2 = \sum_{\xi\in 2\pi\Z} (1+\xi^2)^s |\hat f(\xi)|^2.
\end{align*}

\subsection{Trace and Determinant}

For an operator $A$ on $L^2(\R)$ with continuous integral kernel $K(x,y)$, one may define the trace via
$$
\tr(A) = \int  K(x,x) \,dx,
$$
as was done already by Fredholm.  This definition is somewhat at odds with the trace-ideal theory: (a) trace-class operators need not have continuous kernels and (b) an operator may have a continuous kernel of compact support, yet not be trace class.  See, for example, the discussion on pages~24 and~128 of \cite{MR2154153}.  Nonetheless, this definition is viable for operators that arise as the product of (two or more) Hilbert--Schmidt operators, even though their kernel may not be continuous.  Such operators are automatically of trace class; moreover, the integral over the diagonal has a unique interpretation.  In particular, if $A$ is Hilbert-Schmidt with kernel $K(x,y)\in L^2(\R\times\R)$, then
$$
\tr(A^2) = \iint_{\R^2} K(x,y)K(y,x) \,dx\,dy
$$
is uniquely determined despite the fact that the kernel $K$ is only determined almost everywhere.  Analogously,
\begin{equation}\label{HS norm}
\| A \|_{\mathfrak I_2}^2 =  \tr(A^* A) = \iint_{\R^2} | K(x,y)|^2 \,dx\,dy.
\end{equation}

The estimates in the following lemma may be regarded as special cases of some of the most rudimentary results in the theory of trace ideals (cf. \cite{MR2154153}); nonetheless, we have elected to include proofs for completeness and since the particular results we need admit elementary proofs consonant with the overall spirit of this paper (emphasizing integral kernels rather than operators).

\begin{lemma}\label{L:dumb}
Let $A_i$ denote Hilbert-Schmidt operators on $L^2(\R)$ with integral kernels $K_i(x,y)$.  Then
\begin{align}\label{dumb norm}
\| A_i \|_{\op} &\leq \| A_i \|_{\mathfrak I_2} \\
| \tr( A_1\cdots A_\ell ) | &\leq \prod_{i=1}^\ell \| A_i \|_{\mathfrak I_2}\quad\text{for all integers $\ell\geq 2$.} \label{dumb holder}
\end{align}
\end{lemma}

\begin{proof}
That the Hilbert--Schmidt norm bounds the operator norm follows trivially from the Cauchy--Schwarz inequality:
$$
\biggl| \iint f(x) K(x,y) g(y)\,dx\,dy \biggr|\leq \| K\|_{L^2(\R^2)} \| f\otimes g\|_{L^2(\R^2)} =\| A\|_{\mathfrak I_2} \|f\|_{L^2} \|g\|_{L^2}.
$$

By definition,
\begin{align*}
\tr( A_1\cdots A_\ell ) &= \int\!\!\cdots\!\!\int K_1(x_1,x_2)K_2(x_2,x_3)\cdots K_\ell(x_\ell,x_1)\,dx_\ell\cdots dx_1 \\
&= \int \bigl\langle K(x_1,\cdot), A_2\cdots A_{\ell-1} K(\cdot,x_1)\bigr\rangle \,dx_1.
\end{align*}
Applying the Cauchy--Schwarz inequality, we may then deduce that
\begin{equation}
\begin{aligned}
\text{LHS\eqref{dumb holder}} &\leq \| A_2\cdots A_{\ell-1}\|_\op \int  \| K_1(x_1,z)\|_{L^2_z} \| K_\ell (z, x_1)\|_{L^2_z} dx_1 \\
&\leq \| A_2\cdots A_{\ell-1}\|_\op \| A_1 \|_{\mathfrak I_2} \| A_\ell \|_{\mathfrak I_2}.
\end{aligned}
\label{not super dumb}
\end{equation}
The estimate \eqref{dumb holder} now follows from \eqref{dumb norm}.
\end{proof}

As we shall see in the next section, the $s=-1$ case of Theorem~\ref{T:main} essentially comprises the specialization of the following lemma to a particular choice of kernel.

\begin{lemma}\label{L:C1}
Let $t\mapsto A(t)$ define a $C^1$ curve in $\mathfrak{I}_2$.  Suppose
\begin{align*}
\bigl\| A(t_0) \bigr\|_{\mathfrak I_2} < \tfrac13.
\end{align*}
Then there is a closed neighborhood $I$ of $t_0$ on which the series
\begin{align}\label{Gen alpha}
\alpha(t) := \sum_{\ell=2}^\infty \frac{(-1)^\ell}{\ell} \tr\bigl\{ A(t)^\ell \bigr\}
\end{align}
converges and defines a $C^1$ function with
 \begin{align}\label{Gen alpha dot}
\tfrac{d\ }{dt} \alpha(t) := \sum_{\ell=2}^\infty (-1)^\ell \tr\bigl\{ A(t)^{\ell-1} \tfrac{d\ }{dt} A(t) \bigr\}.
\end{align}
Moreover, if $A(t)$ is self-adjoint, then
\begin{align}\label{comp size}
\tfrac13 \| A(t) \|^2_{\mathfrak I_2} \leq \alpha(t) \leq \tfrac23 \| A(t) \|^2_{\mathfrak I_2} \quad\text{for all $t\in I$.}
\end{align}
\end{lemma}

\begin{proof}
We choose as $I$ any interval containing $t_0$ on which
\begin{align}\label{I hype}
\bigl\| A(t) \bigr\|_{\mathfrak I_2} \leq \tfrac13.
\end{align}

The estimate \eqref{dumb holder} shows that the series \eqref{Gen alpha} and \eqref{Gen alpha dot} converge as soon as
$$
\bigl\| A(t) \bigr\|^2_{\mathfrak I_2} < 1,
$$
and so throughout the interval $I$.  Because of our stronger hypothesis \eqref{I hype}, we may even conclude
\begin{equation}
 \Bigl| \alpha(t) -  \tfrac12\tr\Bigl\{ A(t)^2 \Bigr\} \Bigr| \leq  \sum_{\ell=3}^\infty \tfrac{1}{\ell} \| A(t) \|^\ell_{\mathfrak I_2} \leq \tfrac16 \| A(t) \|^2_{\mathfrak I_2}.
\end{equation}
Equation \eqref{comp size} is just a recapitulation of this.

The uniform convergence exhibited above also guarantees that $\alpha(t)$ is differentiable and that the series \eqref{Gen alpha dot} converges to its derivative; see, for example \cite[\S9.10]{MR0344384}.
\end{proof}

\section{First conservation laws for KdV}\label{S:2}

\begin{prop}\label{P:free}  The free resolvent admits the following explicit kernel:
\begin{align}
R_0(x,y;-\kappa^2) &= \tfrac{1}{2\kappa} e^{-\kappa|x-y|} \quad\text{on $\R$} \label{R resolvent}\\
R_0(x,y;-\kappa^2) &= \tfrac{1}{2\kappa} (1-e^{-\kappa})^{-1} \bigl[ e^{-\kappa|x-y|} + e^{-\kappa+\kappa|x-y|}\bigr]\quad\text{on $\R/\Z$} \label{S resolvent}
\end{align}
where, in the circle case, $|x-y|=\dist(x-y,\Z)$.  Correspondingly,
\begin{align}
\Bigl\| \sqrt{R_0}\, q\, \sqrt{R_0} \Bigr\|^2_{\mathfrak I_2(\R)} &= \frac1\kappa \int \frac{|\hat q(\xi)|^2}{\xi^2+4\kappa^2}\,d\xi \label{R I2}\\
\Bigl\| \sqrt{R_0}\, q\, \sqrt{R_0} \Bigr\|^2_{\mathfrak I_2(\R/\Z)} &= \frac{2 e^{-\kappa}}{(1-e^{-\kappa})^2} \frac{|\hat q(0)|^2}{4\kappa^2} \notag\\
&\quad + \frac{1-e^{-2\kappa}}{\kappa(1-e^{-\kappa})^2}\sum_{\xi\in2\pi\Z} \frac{|\hat q(\xi)|^2}{\xi^2+4\kappa^2} \label{C I2}
\end{align}
for any $q\in H^{-1}$.
\end{prop}

\begin{proof}
Let us first consider the line case. The formula \eqref{R resolvent} is well-known and can be easily confirmed.  This formula shows us that
\begin{align}\label{R Rsq}
[ R_0(x,y;-\kappa^2) ]^2 &= \tfrac{1}{\kappa} R_0(x,y;-4\kappa^2)
\end{align}
and correspondingly,
\begin{align*}
\Bigl\| \sqrt{R_0}\, q\, \sqrt{R_0} \Bigr\|^2_{\mathfrak I_2(\R)} &= \frac1{\kappa} \iint q(x)R_0(x,y;-4\kappa^2)q(y)\,dx\,dy = \text{RHS\eqref{R I2}},
\end{align*}
at least for Schwartz $q$.  The result for general $q$ follows by approximation.

The identity \eqref{S resolvent} is easily verified.  One quick way of deriving this formula is from \eqref{R resolvent} via the method of images, according to which,
$$
R_0(x,y;-\kappa^2) = \sum_{\ell\in\Z} \tfrac{1}{2\kappa} e^{-\kappa|x-y-\ell|} \quad\text{on $\R/\Z$}.
$$

The analogue of \eqref{R Rsq} in the circle setting is
\begin{align}\label{C Rsq}
[ R_0(x,y;-\kappa^2) ]^2 &= \frac{1-e^{-2\kappa}}{\kappa(1-e^{-\kappa})^2} R_0(x,y;-4\kappa^2) + \frac{e^{-\kappa}}{2\kappa^2(1-e^{-\kappa})^2},
\end{align}
from which \eqref{C I2} readily follows.
\end{proof}

To facilitate treating the line and circle cases simultaneously, we recast the preceding exact formulae in a simpler form:

\begin{corollary}\label{C:I2} If $q\in H^{-1}$ on $\R$ or $\R/\Z$ and  $\kappa\geq 1$, then
\begin{gather}\label{H-1kappa}
\tfrac{1}{\kappa}  \bigl\langle q, (-\partial_x^2+4\kappa^2)^{-1} q\bigr\rangle \leq   \bigl\| \sqrt{R_0}\, q\, \sqrt{R_0}\, \bigr\|^2_{\mathfrak I_2} \leq \tfrac{5}{\kappa} \bigl\langle q, (-\partial_x^2+4\kappa^2)^{-1} q\bigr\rangle
\end{gather}
and consequently,
\begin{gather}
\label{H-1normal}
\tfrac1{4\kappa^3} \| q \|_{H^{-1}}^2 \leq   \bigl\| \sqrt{R_0}\, q\, \sqrt{R_0}\, \bigr\|^2_{\mathfrak I_2} \leq \tfrac{5}{\kappa} \| q \|_{H^{-1}}^2.
\end{gather}
\end{corollary}

We are now ready to show conservation of the logarithm of the renormalized perturbation determinant:

\begin{prop}\label{P:KdVcons}
Let $q(t)$ be a Schwartz solution to KdV.  Then
\begin{equation*}
\ddt \alpha(\kappa;q(t)) = 0
\end{equation*}
for all $\kappa$ obeying $\kappa \geq 1 + 45 \| q(t) \|_{H^{-1}}^2$.
\end{prop}

\begin{remark}
As the perturbation determinant is an analytic function of $\kappa$ in the upper half-plane, constancy extends to this whole region.
\end{remark}

\begin{proof}
The bounds \eqref{H-1normal} show that Lemma~\ref{L:C1} applies.  Thus
\begin{align}\label{dt KdValpha}
\frac{d\ }{dt} \alpha(\kappa;q(t)) = \sum_{\ell=2}^\infty (-1)^\ell \tr\Bigl\{ \bigl( \sqrt{R_0}\, q\, \sqrt{R_0}\bigr)^{\ell-1} \sqrt{R_0}\, \tfrac{dq}{dt}\, \sqrt{R_0} \Bigr\}.
\end{align}

In view of the above and the fact that $q(t)$ is Schwartz, it suffices to show the following:
\begin{align}
\tr\Bigl\{(R_0 q)^{l-1} R_0  \bigl(-q''')\Bigr\}
	 &= \tr\Bigl\{(R_0 q)^{l-2} R_0\bigl(6 q q'\bigr)\Bigr\}, \quad\forall \ell\geq 2\label{E:telescope}\\
\tr\bigl\{  R_0 \bigl(6 q q'\bigr)\bigr\}&=0, \label{E:eyepiece}
\end{align}
which is what we will do.

Just as $q'=[\partial,q]$, so
\begin{align*}
-q''' &= -[\partial,[\partial,[\partial,q]]] = -[\partial,\partial^2 q + q\partial^2 - 2\partial q\partial] \\
&= (-\partial^2+\kappa^2)q'+ q' (-\partial^2+\kappa^2) - 2(-\partial^2+\kappa^2)q\partial +  2\partial q(-\partial^2+\kappa^2) \\
&\quad - 4\kappa^2[\partial,q].
\end{align*}
Substituting this into LHS\eqref{E:telescope} and cycling the trace yields
\begin{align*}
\text{LHS\eqref{E:telescope}} &= \tr\Bigl\{ (R_0 q)^{l-2} R_0\bigl(2qq'+2[\partial,q^2])\Bigr\}
    - 4\kappa^2 \tr\Bigl\{ (R_0 q)^{l-1} R_0  [\partial,q] \Bigr\} \\
&= \text{RHS\eqref{E:telescope}} - 0.
\end{align*}
To see that the second trace vanishes, one should commute $R_0$ and $\partial$ and then cycle the trace. One is left taking the trace of the zero operator.

Writing the trace in \eqref{E:eyepiece} as an integral of the kernel over the diagonal, the requisite vanishing follows from the constancy of $G_0(x,x)$ and the fact that $6qq'$ is a complete derivative.
\end{proof}

We have several other proofs of this proposition, but felt this one is the most elementary.  While we speak of commutators and cycling the trace, this merely represents a compact means of expressing more elementary operations, such as integration by parts and the application of Fubini's theorem.

It is tempting to believe that this proposition might follow simply from \eqref{unitary} by some abstract means, or more generally, that (when defined) the perturbation determinant between unitarily equivalent operators is unity.  This is not true.  We will now demonstrate this fallacy in a manner relevant to the KdV hierarchy with step-like initial data:  Let $v(x)$ be a function for which $v'(x)<0$ and belongs to Schwartz class and define
$$
q(t,x) = v(x+t),
$$
which is the evolution associated to the (commuting) Hamiltonian $\int \frac12|q|^2\,dx$.  Then despite the fact that the Schr\"odinger operators
$$
-\partial^2 + q(t,x)
$$
are unitarily equivalent for all values of $t$, the associated perturbation determinant obeys
\begin{align*}
\log\det\bigl[ 1 + (-\partial^2 + q(0) + \kappa^2)^{-1/2}(q(t)-q(0))(-\partial^2 + q(0) + \kappa^2)^{-1/2}\bigr] >0
\end{align*}
for $t>0$ and $\kappa$ sufficiently large, because every term in the associated series is positive.

The foregoing leads immediately to the following special cases of Theorem~\ref{T:main}; we will take up the remaining cases and the study of Besov norms in the next section.

\begin{theorem}\label{s<0} Let $q(t)$ be a Schwartz solution to KdV on $\R$ or $\R/\Z$.  Then
\begin{equation}\label{Key Est}
\Bigl\| \sqrt{R_0}\, q(t)\, \sqrt{R_0} \Bigr\|^2_{\mathfrak I_2} \leq 2 \Bigl\| \sqrt{R_0}\, q(0)\, \sqrt{R_0} \Bigr\|^2_{\mathfrak I_2} < \tfrac19
\end{equation}
for all $\kappa\geq 1 + 90 \| q(0) \|_{H^{-1}}^2$.  Moreover,
$$
\| q(t) \|_{H^{s}} \lesssim  \bigl(1 + \| q(0) \|_{H^{-1}}^2 \bigr)^{|s|} \| q(0) \|_{H^{s}},
$$
which shows that Theorem~\ref{T:main} holds for $-1\leq s<0$.
\end{theorem}

\begin{proof}
In view of \eqref{H-1normal}, our hypothesis on $\kappa$ guarantees that
$$
\bigl\| \sqrt{R_0}\, q(0)\, \sqrt{R_0}\, \bigr\|^2_{\mathfrak I_2} < \tfrac{1}{18} < \tfrac19
$$
and so Lemma~\ref{L:C1} applies.  Moreover, by conservation of $\alpha$ we then have
$$
\bigl\| \sqrt{R_0}\, q(t)\, \sqrt{R_0}\, \bigr\|^2_{\mathfrak I_2} \leq 3\alpha(\kappa; q(t)) \leq 2\bigl\| \sqrt{R_0}\, q(0)\, \sqrt{R_0} \,\bigr\|^2_{\mathfrak I_2} < \tfrac19
$$
in a neighborhood of $t=0$. A simple continuity argument then completes the proof of \eqref{Key Est}.

Choosing $\kappa=\kappa_0:=1 + 90 \| q(0) \|_{H^{-1}}^2$ in \eqref{Key Est} and invoking Corollary~\ref{C:I2}, it follows that
$$
\| q(t) \|_{H^{-1}}^2 \leq  40 \kappa_0^2 \| q(0) \|_{H^{-1}}^2.
$$
This yields the upper bound in Theorem~\ref{T:main} in the case $s=-1$.  The lower bound follows from this together with time translation symmetry.

Let us now address the case $-1<s<0$.  From \eqref{H-1kappa} and \eqref{Key Est}, we deduce
$$
\bigl\langle q(t), (-\partial_x^2+4\kappa^2)^{-1} q(t)\bigr\rangle \leq 10 \bigl\langle q(0), (-\partial_x^2+4\kappa^2)^{-1} q(0)\bigr\rangle \qquad\forall \kappa\geq \kappa_0.
$$
Integrating both sides against the measure $\kappa^{1+2s} \,d\kappa$ over the interval $[\kappa_0,\infty)$ and using the relation
\begin{align}\label{secIntegral}
(\xi^2+\kappa_0^2)^s \sim \int_{\kappa_0}^\infty  \frac{1}{\xi^2+4\kappa^2} \kappa^{1+2s} \,d\kappa
\end{align}
(which holds with absolute constants depending only on $s$), we deduce that
$$
\bigl\langle q(t), (-\partial_x^2+4\kappa_0^2)^{s} q(t)\bigr\rangle \lesssim \bigl\langle q(0), (-\partial_x^2+4\kappa_0^2)^{s} q(0)\bigr\rangle
$$
and hence that
$$
\| q(t) \|_{H^{s}}^2 \lesssim \kappa_0^{2|s|}  \| q(0) \|_{H^{s}}^2 \lesssim \bigl(1 + \| q(0) \|_{H^{-1}}^2 \bigr)^{2|s|} \| q(0) \|_{H^{s}}^2.
$$
The upper bound in Theorem~\ref{T:main} follows immediately from this; the lower bound can then deduced by invoking time translation symmetry.
\end{proof}

\section{Conservation of other norms}\label{S:3}

In this section we treat two main topics: Sobolev spaces at positive regularity and $L^2$-based Besov norms.

\begin{definition}\label{D:Besov}
Given $s\in\R$ and $1\leq r\leq \infty$, we define the \emph{Besov norm}
\begin{align*}
\| f \|_{B^{s,2}_r} = \biggl[ \| \hat f(\xi) \|^r_{L^2(|\xi|\leq 1)} + \sum_{N\in\NN}  N^{rs} \| \hat f(\xi) \|^r_{L^2(N< |\xi|\leq 2N)} \biggr]^{1/r},
\end{align*}
with the usual interpretation when $r=\infty$.  Here the sum in $N$ is taken over $\NN:=\{1,2,4,8,16,\ldots\}$.  In the line case, $L^2$ refers to integration against Lebesgue measure; in the circle case, we use counting measure on $2\pi\Z$.
\end{definition}

As we are working with $L^2$-based Besov norms, replacement of the sharp Fourier cutoffs used above by regular Littlewood--Paley projections yields an equivalent norm (cf. Lemmas~\ref{L:Besov1},~\ref{L:Besov2}, and~\ref{L:Besov3} below).  We have elected to use sharp cutoffs in this paper in order to keep the presentation more elementary.

We first obtain bounds on Besov norms in the range $-1<s<0$, which requires only very simple modifications to the ideas presented already.  The key is to connect such Besov norms to the quantities in \eqref{H-1kappa} and then invoke \eqref{Key Est}.  The first step is covered by the following lemma:

\begin{lemma}\label{L:Besov1}
Fix $-1<s<0$, $r\in[1,\infty]$, and $\kappa_0\geq 1$. For any Schwartz function $f$ on $\R$ or $\R/\Z$,
\begin{align}\label{E:Besov1}
\| f \|_{B^{s,2}_r} \sim \biggl[ \sum_{N\in\NN} N^{rs} \Bigl\langle f, \tfrac{\kappa_0^2 N^2}{-\partial_x^2 + 4\kappa_0^2 N^2} f\Bigr\rangle ^{r/2} \biggr]^{1/r}
\end{align}
with implicit constants depending only on $s$ and $\kappa_0$.
\end{lemma}

\begin{proof}
By Plancherel,
$$
\| \hat f(\xi) \|^2_{L^2(|\xi|\leq N)} \leq 5 \kappa_0^2 N^2 \bigl\langle f, (-\partial_x^2 + 4\kappa_0^2 N^2)^{-1} f\bigr\rangle
$$
and consequently,
\begin{align}
\text{LHS\eqref{E:Besov1}} \leq \sqrt{5} \cdot \text{RHS\eqref{E:Besov1}}.
\label{E:Besov1a}
\end{align}

Towards the other direction, we note that for $N\in\NN$,
\begin{align*}
\kappa_0^2 N^2 \bigl\langle f,\, & (-\partial_x^2 + 4\kappa_0^2 N^2)^{-1} f\bigr\rangle \\
&\leq \| \hat f(\xi) \|^2_{L^2(|\xi|\leq 1)} + \sum_{M\in\NN}  \tfrac{\kappa_0^2N^2}{M^2 + 4\kappa_0^2 N^2} \| \hat f(\xi) \|^2_{L^2(M< |\xi|\leq 2M)} \\
&\leq \biggl[ \| \hat f(\xi) \|_{L^2(|\xi|\leq 1)} + \sum_{M\in\NN}  \tfrac{\kappa_0N}{\sqrt{M^2 + 4\kappa_0^2 N^2}} \| \hat f(\xi) \|_{L^2(M< |\xi|\leq 2M)} \biggr]^2,
\end{align*}
from which it follows that
\begin{align*}
\text{RHS\eqref{E:Besov1}} \leq \biggl\| N^s  \| \hat f(\xi) \|_{L^2(|\xi|\leq 1)}
	+ \! \sum_{M\in\NN} \! \tfrac{\kappa_0 N^{1+s}}{\sqrt{M^2 + 4\kappa_0^2 N^2}} \| \hat f(\xi) \|_{L^2(M<|\xi|\leq 2M)} \biggr\|_{\ell^r(\NN)} .
\end{align*}
In this way, the proof of the remaining inequality is reduced to showing that a certain matrix defines a bounded operator on $\ell^r$.  To verify this, we use Schur's test: The row sums are bounded by
$$
N^s + \sum_{M\in\NN}  \tfrac{\kappa_0N^{1+s} M^{-s}}{\sqrt{M^2 + 4\kappa_0^2 N^2}} \lesssim_s 1 + \kappa_0^{-s} \lesssim_s \kappa_0^{-s}
$$
uniformly in $N$, while we bound the column sums by
$$
\sum_{N\in\NN} N^s\lesssim_s 1 \qtq{and} \sum_{N\in\NN} \tfrac{\kappa_0N^{1+s} M^{-s}}{\sqrt{M^2 + 4\kappa_0^2 N^2}} \lesssim_s \kappa_0^{-s}
$$
uniformly in $M$.  Correspondingly,
\begin{align}
\text{RHS\eqref{E:Besov1}} \lesssim_s \kappa_0^{|s|} \cdot \text{LHS\eqref{E:Besov1}},
\label{E:Besov1b}
\end{align}
which completes the proof of \eqref{E:Besov1}.
\end{proof}

\begin{corollary}\label{C:neg}
Fix $-1<s<0$ and $r\in[1,\infty]$.  For any Schwartz solution $q(t)$ to KdV on $\R$ or $\R/\Z$, we have
$$
\| q(t) \|_{B^{s,2}_r} \lesssim \| q(0) \|_{B^{s,2}_r} \bigl( 1 + \| q(0) \|_{H^{-1}}^2\bigr)^{|s|} \lesssim \| q(0) \|_{B^{s,2}_r} \bigl( 1 + \| q(0) \|_{B^{s,2}_r}^2\bigr)^{|s|}.
$$
\end{corollary}

\begin{proof}
Choosing $\kappa_0=1+90\|q(0)\|_{H^{-1}}^2$, Theorem~\ref{s<0} and \eqref{H-1kappa} together imply
$$
\bigl\langle q(t), (-\partial_x^2 + 4\kappa_0^2 N^2)^{-1} q(t)\bigr\rangle \leq 10 \bigl\langle q(0), (-\partial_x^2 + 4\kappa_0^2 N^2)^{-1} q(0)\bigr\rangle
$$
for all $N\geq 1$. Summing in $N\in\NN$ as in \eqref{E:Besov1} and applying \eqref{E:Besov1a} and \eqref{E:Besov1b} to the left and right sides, respectively, then yields
$$
\| q(t) \|_{B^{s,2}_r} \lesssim \| q(0) \|_{B^{s,2}_r} \bigl( 1 +  90 \| q(0) \|_{H^{-1}}^2\bigr)^{|s|}.
$$
The result now follows since $B^{s,2}_{r\vphantom{2}} \hookrightarrow B^{-1,2}_2 \cong H^{-1}$ as follows readily from the definition.
\end{proof}

Lemma~\ref{L:Besov1} does not extend to Besov norms with $s=-1$ because the function
\begin{align*}
\xi \mapsto \frac{\kappa^2}{\xi^2 + 4\kappa^2}
\end{align*}
decays too slowly as $\xi\to\infty$.  On the other hand, to extend the argument into the region $s\geq 0$, we would need a function decaying suitably as $\xi\to 0$.

Our chosen remedy for these problems is to take linear combinations of these function at different values of $\kappa$ to cancel any undesirable behavior.  This approach generates one new problem, namely, that \eqref{Key Est} does not extend to such linear combinations (which necessarily have both positive and negative coefficients), because \eqref{comp size} does not extend to such a setting.  As we will see in due course, the cure for this second new problem is to bound the difference between $\alpha(t)$ and the first term in the series \eqref{Gen alpha} in a different way --- one which exploits Theorem~\ref{s<0}.  Let us begin with the analogue of Lemma~\ref{L:Besov1} for the case $s=-1$.

\begin{lemma}\label{L:Besov2}
Fix $1\leq r\leq \infty$, $\kappa_0\geq 1$, and
\begin{align}\label{w-1}
w(\xi;\kappa) = 4 \frac{(\kappa/2)^2}{\xi^2 + \kappa^2} - \frac{\kappa^2}{\xi^2 + 4\kappa^2} = \frac{3\kappa^4}{4(\xi^2 + \kappa^2)(\xi^2 + 4\kappa^2)}.
\end{align}
Then for any Schwartz function $f$ on $\R$ or $\R/\Z$,
\begin{align}\label{E:Besov2}
\| f \|_{B^{-1,2}_r}  \sim \biggl[ \sum_{N\in\NN} N^{-r} \bigl\langle f, \,w(-i\partial_x,\kappa_0 N) f\bigr\rangle^{r/2} \biggr]^{1/r}.
\end{align}
\end{lemma}

\begin{proof}
As in the proof of Lemma~\ref{L:Besov1},
$$
\| \hat f(\xi) \|^2_{L^2(|\xi|\leq N)} \leq \tfrac{40}3 \langle f, \,w(-i\partial_x,\kappa_0 N) f\bigr\rangle
$$
and consequently,
\begin{align}
\text{LHS\eqref{E:Besov2}} \leq \bigl(\tfrac{40}3 \bigr)^{1/2} \cdot \text{RHS\eqref{E:Besov2}}.
\label{E:Besov2a}
\end{align}
Continuing to argue as in that proof shows
\begin{align*}
\text{RHS\eqref{E:Besov2}}  \lesssim_r \biggl\|\tfrac1N \| \hat f(\xi) \|_{L^2(|\xi|\leq 1)}
	+ \! \sum_{M\in\NN} \! \tfrac{\kappa_0^2 N}{M^2 + 4\kappa_0^2 N^2} \| \hat f(\xi) \|_{L^2(M<|\xi|\leq 2M)} \biggr\|_{\ell^r(\NN)}
\end{align*}
and then that
\begin{align}
\text{RHS\eqref{E:Besov2}} \lesssim_r \kappa_0 \cdot \text{LHS\eqref{E:Besov2}}.
\label{E:Besov2b}
\end{align}
This completes the proof of the lemma.
\end{proof}

Repeating the same arguments one more time reveals the following:

\begin{lemma}\label{L:Besov3}
Fix $1\leq r\leq \infty$ and $-1<s<1$ and define
\begin{align}\label{w>0}
w(\xi;\kappa) = \frac{\kappa^2}{\xi^2 + 4\kappa^2} - \frac{(\kappa/2)^2}{\xi^2 + \kappa^2} = \frac{3\kappa^2\xi^2}{(\xi^2 + \kappa^2)(\xi^2 + 4\kappa^2)}.
\end{align}
Then
\begin{align}
\| f \|_{B^{s,2}_r} \lesssim_s \|f\|_{H^{-1}} +  \kappa_0 \biggl( \sum_{N\in\NN} N^{rs} \langle f, \,w(-i\partial_x,\kappa_0 N) f\bigr\rangle^{r/2} \biggr)^{1/r}
\label{E:Besov3a}
\end{align}
and
\begin{align}
\biggl( \sum_{N\in\NN} N^{rs} \langle f, \,w(-i\partial_x,\kappa_0 N) f\bigr\rangle^{r/2} \biggr)^{1/r} \lesssim \kappa_0^{-s} \| f \|_{B^{s,2}_r}
\label{E:Besov3b}
\end{align}
uniformly for $\kappa_0\geq 1$.
\end{lemma}

As mentioned earlier, the second ingredient in our argument is to use Theorem~\ref{s<0} to upgrade \eqref{Key Est} to a two-sided estimate.  This is encapsulated in the following proposition:

\begin{prop}\label{P:D}
Let $q$ be a Schwartz solution of KdV and let
\begin{align*}
D(t;\kappa):= \kappa^3 \biggl| \tr\Bigl\{ \bigl( \sqrt{R_0}\, q(t)\, \sqrt{R_0}\;\!\bigr)^2 \Bigr\} - \tr\Bigl\{ \bigl( \sqrt{R_0}\, q(0)\, \sqrt{R_0}\;\!\bigr)^2 \Bigr\} \biggr|.
\end{align*}
Then for any $\kappa\geq 1 + 90 \| q(0) \|_{H^{-1}}^2$,
\begin{align}\label{E:Dest1}
D(t;\kappa) \lesssim  \kappa^{-\frac32 -3\sigma}  \|q(0)\|_{H^{\sigma}}^3 \qtq{for all}  -1 \leq \sigma \leq 0.
\end{align}
Moreover, 
\begin{align}\label{E:Dest2}
D(t;\kappa) \lesssim \kappa^{-2} \Bigl[ \|q(t)\|_{L^\infty} + \|q(0)\|_{L^\infty} \Bigr] \|q(0)\|_{L^2}^2 .
\end{align}
\end{prop}

\begin{proof}
From \eqref{KdValpha}, we obtain
\begin{align*}
D(t;\kappa) &\leq 2\kappa^3|\alpha(\kappa;q(t))- \alpha(\kappa;q(0))| \\
&\qquad + \sum_{\ell=3}^\infty \tfrac{2\kappa^3}{\ell} \biggl| \tr\Bigl\{ \bigl( \sqrt{R_0}\, q(t)\, \sqrt{R_0}\;\!\bigr)^\ell \Bigr\} - \tr\Bigl\{ \bigl( \sqrt{R_0}\, q(0)\, \sqrt{R_0}\;\!\bigr)^\ell \Bigr\} \biggr|.
\end{align*}
Proposition~\ref{P:KdVcons} guarantees that $\alpha(\kappa;q(t))$ is conserved in time.  Thus, by \eqref{not super dumb}, \eqref{dumb norm}, and \eqref{Key Est},
\begin{align}\label{knee}
D(t;\kappa) &\leq  \tfrac{4\kappa^3}{3} \bigl\| \sqrt{R_0}\, q(0)\, \sqrt{R_0}\;\!\bigr\|_{\mathfrak I_2}^2 \\
&\quad\ \ \times \Bigl[ \bigl\| \sqrt{R_0}\, q(t)\, \sqrt{R_0}\;\!\bigr\|_\op + \bigl\| \sqrt{R_0}\, q(0)\, \sqrt{R_0}\;\!\bigr\|_\op\Bigr] \sum_{\ell=3}^\infty \bigl(\tfrac13\bigr)^{\ell-3}. \notag
\end{align}
To deduce \eqref{E:Dest1} from here, we may bound the operator norm by the Hilbert--Schmidt norm and apply \eqref{Key Est} and then \eqref{H-1kappa} to obtain
\begin{align*}
D(t;\kappa) &\lesssim \kappa^{3/2}\bigl\langle q(0), (-\partial_x^2 + 4\kappa^2)^{-1} q(0)\bigr\rangle^{3/2},
\end{align*}
from which \eqref{E:Dest1} follows immediately. To obtain \eqref{E:Dest2}, we employ the bound
\begin{equation*}
\bigl\| \sqrt{R_0}\, q\, \sqrt{R_0}\;\!\bigr\|_\op \leq \kappa^{-2} \|q\|_{L^\infty},
\end{equation*}
which is merely the product of the norms of each operator.
\end{proof}

\begin{theorem}\label{T:Besov}
Let $q$ be a Schwartz solution of KdV.  Then
\begin{align}\label{E:BT}
\| q(t) \|_{B^{s,2}_r} \lesssim \| q(0) \|_{B^{s,2}_r} \bigl( 1 + \| q(0) \|_{B^{s,2}_r}^2\bigr),
\end{align}
both when $s=-1$ and $1\leq r\leq 2$ and  when $-1<s<1$ and $1\leq r\leq \infty$.
\end{theorem}

\begin{proof}
This result has already been proven when $s=-1$ and $r=2$, as well as when $-1<s<0$ and $1\leq r\leq\infty$; see Theorem~\ref{s<0} and Corollary~\ref{C:neg}, respectively.  We will only consider the remaining cases here.

Throughout the proof we assume that
$$
\kappa \geq \kappa_0 = 1 + 90 \| q(0) \|_{H^{-1}}^2.
$$
We will also make repeated use of the following:
\begin{equation}\label{C diff}
\biggl| \kappa^3 \tr\Bigl\{ \bigl( \sqrt{R_0}\, q(t)\, \sqrt{R_0}\;\!\bigr)^2 \Bigr\}  - \bigl\langle q(t), \tfrac{ \kappa^2}{-\partial_x^2 + 4\kappa^2} q(t)\bigr\rangle \biggr| \lesssim e^{-\kappa/2} \| q(0) \|_{H^{-1}}^2,
\end{equation}
which follows from \eqref{R I2}, \eqref{C I2}, and \eqref{Key Est}.  Indeed, this difference is zero in the line case.

Let us begin with the cases $s=-1$ and $1\leq r<2$.  Choosing $w$ as in \eqref{w-1}, setting $\sigma=-1$ in \eqref{E:Dest1} and applying \eqref{C diff}, we have
\begin{align*}
\bigl\langle q(t), \,w(-i\partial_x,\kappa) q(t)\bigr\rangle  &\leq \bigl\langle q(0), \,w(-i\partial_x,\kappa) q(0) \bigr\rangle + O\bigl( e^{-\kappa/2} \|q(0)\|_{H^{-1}}^2 \bigr) \\
&\qquad + O\bigl( \kappa^{3/2} \|q(0)\|_{H^{-1}}^3 \bigr).
\end{align*}
But then by \eqref{E:Besov2a} and \eqref{E:Besov2b},
\begin{align*}
\| q(t) \|_{B^{-1,2}_r} &\lesssim \kappa_0 \| q(0) \|_{B^{-1,2}_r} + \|q(0)\|_{H^{-1}} + \kappa_0^{3/4} \|q(0)\|_{H^{-1}}^{3/2} \\
&\lesssim \| q(0) \|_{B^{-1,2}_r} \bigl( 1 + \| q(0) \|_{H^{-1}}^2 \bigr),
\end{align*}
from which \eqref{E:BT} follows.

Consider now the case $0\leq s<\frac34$ and $r$ arbitrary.  Proceeding directly as above, but choosing $w$ as in \eqref{w>0} yields
\begin{align*}
\bigl\langle q(t), \,w(-i\partial_x,\kappa) q(t)\bigr\rangle  &\leq \bigl\langle q(0), \,w(-i\partial_x,\kappa) q(0) \bigr\rangle + O\bigl( e^{-\kappa/2} \|q(0)\|_{H^{-1}}^2 \bigr) \\
&\qquad + O\bigl( \kappa^{-\frac32-3\sigma} \|q(0)\|_{H^\sigma}^3 \bigr),
\end{align*}
provided $-1\leq \sigma\leq 0$.  As $s<\frac34$, we may choose $\sigma\in[-1,0]$ so that $s<\tfrac34+\tfrac{3\sigma}2$, which allows us to sum the contribution of the last error term.  In this way, \eqref{E:Besov3a} and \eqref{E:Besov3b} yield
\begin{align*}
\| q(t) \|_{B^{s,2}_r} &\lesssim \|q(t)\|_{H^{-1}} + \kappa_0^{1-s} \| q(0) \|_{B^{s,2}_r} + \|q(0)\|_{H^{-1}} + \kappa_0^{\frac14-\frac{3\sigma}2} \|q(0)\|_{H^\sigma}^{3/2} \\
&\lesssim \| q(0) \|_{B^{s,2}_r} \bigl( 1 + \| q(0) \|_{H^{-1}}^2\bigr) + \kappa_0^{\frac14-\frac{3\sigma}2} \|q(0)\|_{H^\sigma}^{3/2}.
\end{align*}
(Note the use of the $s=-1$, $r=2$ result here.)  It is relatively easy to deduce \eqref{E:BT} from here.

To address the remaining case $\frac34\leq s<1$, we combine the preceding argument with \eqref{E:Dest2} and the estimate
$$
\| q(t) \|_{L^\infty} \lesssim \| \hat q(t) \|_{L^1} \lesssim \| q(t) \|_{B^{\frac12,2}_1} \lesssim  \| q(0) \|_{B^{\frac12,2}_1} \bigl( 1 + \| q(0) \|_{H^{-1}}^2\bigr).
$$
Proceeding in the manner just described, we see that
\begin{align*}
\bigl\langle q(t), \,w(-i\partial_x,\kappa) q(t)\bigr\rangle  &\leq \bigl\langle q(0), \,w(-i\partial_x,\kappa) q(0) \bigr\rangle + O\bigl( e^{-\kappa/2} \|q(0)\|_{L^2}^2 \bigr) \\
&\qquad + O\Bigl( \kappa^{-2} \| q(0) \|_{B^{\frac12,2}_1} \| q(0) \|_{L^2}^2 \bigl( 1 + \| q(0) \|_{H^{-1}}^2\bigr)\Bigr)
\end{align*}
and hence that
\begin{align*}
\| q(t) \|_{B^{s,2}_r} &\lesssim  \| q(0) \|_{B^{s,2}_r} \bigl( 1 + \| q(0) \|_{H^{-1}}^2\bigr) \\
&\qquad  + \| q(0) \|_{B^{\frac12,2}_1}^{1/2}  \| q(0) \|_{L^2} \bigl( 1 + \| q(0) \|_{H^{-1}}^2\bigr)^{1/2} \\
&\lesssim \| q(0) \|_{B^{s,2}_r} \bigl( 1 + \| q(0) \|_{L^2}^2\bigr),
\end{align*}
which is evidently stronger than the claimed \eqref{E:BT}.
\end{proof}

\section{AKNS/ZS examples}\label{S:4}

Many completely integrable PDE admit a Lax pair of the following form:
$$
\tfrac{d\ }{dt}  L(t;\kappa) = [P(t;\kappa), L(t;\kappa)] \qtq{with}
	L(t;\kappa) = \begin{bmatrix}-\partial+\kappa & i q(x) \\ \mp i\bar q(x) & -\partial-\kappa \end{bmatrix}
$$
and some operator pencil $P(t;\kappa)$.  The names AKNS and ZS originate in the surnames of the authors of \cite{MR0450815} and \cite{ZS}, respectively.  Examples of models that lie within this framework include the cubic nonlinear Schr\"odinger equation
\begin{align}
- i \tfrac{d\ }{dt} q &= -q'' \pm 2 |q|^2 q \tag{NLS}\label{NLS},
\end{align}
the modified KdV equation of Hirota \cite{Hirota}
\begin{align}
\tfrac{d\ }{dt} q &= - q''' \pm 6 |q|^2 q'   \tag{HmKdV}\label{HmKdV},
\end{align}
as well as the Sasa--Satsuma mKdV equation \cite{MR1104387}, and the sin-Gordon and sinh-Gordon equations.  Note that the Hirota and Sasa--Satsuma equations describe the evolution of a \emph{complex}-valued field.  They are distinct generalizations of the traditional mKdV equation
\begin{equation}
\tfrac{d\ }{dt} q = - q''' \pm 6 q^2 q',   \tag{mKdV}\label{mKdV}
\end{equation}
which is posed for a real-valued field.  We do not give explicit expressions for the operator pencils $P(t;\kappa)$ related to any of these models, because they will play no role in what follows.  Rather, we proceed as we did in the KdV case and give a direct proof of the conservation of the perturbation determinant.  For simplicity, we will restrict the exposition to the line case.

By analogy with what has gone before, we define
\begin{align}\label{AKNS alpha}
\alpha(\kappa;q) := \Re \sum_{\ell=1}^\infty \frac{(\mp 1)^{\ell-1}}{\ell} \tr\Bigl\{\bigl[(\kappa-\partial)^{-1/2} q (\kappa+\partial)^{-1} \bar q (\kappa-\partial)^{-1/2} \bigr]^\ell\Bigr\},
\end{align}
which formally represents
$$
\pm \Re \log\det\left( \begin{bmatrix} (-\partial+\kappa)^{-1} & 0 \\ 0 & (-\partial-\kappa)^{-1} \end{bmatrix}
	\begin{bmatrix}-\partial+\kappa & i q(x) \\ \mp i\bar q(x) & -\partial-\kappa \end{bmatrix} \right).
$$
Note that we do not renormalize the determinant here --- it is already quadratic in $q$.  The operators $(\kappa\pm\partial)^{-1}$ appearing here are defined via the Fourier transform and so exist for every $\kappa>0$.  We further define their square-roots via
$$
[(\kappa\pm\partial)^{-1/2} f]\widehat{\ }(\xi) = \frac1{\sqrt{\kappa \pm i \xi}} \,\hat f(\xi),
$$
where the complex square-root is determined via continuity and $\sqrt{\kappa}>0$.

Our first task is to guarantee convergence of the series defining $\alpha(\kappa)$ for $\kappa$ sufficiently large. This follows from our next lemma, because the operators
$$
 (\kappa-\partial)^{-1/2} q (\kappa+\partial)^{-1/2} \qtq{and} (\kappa+\partial)^{-1/2} \bar q (\kappa-\partial)^{-1/2}
$$
are intertwined by the unitary operator $U:f(x)\to f(-x)$, up to the replacement of $q(x)$ by $\bar q(-x)$.

\begin{lemma}\label{L:AKNS I2}
For $\kappa>0$ and $q(x)$ Schwartz,
\begin{equation}\label{AKNS I2}
\Bigl\|(\kappa-\partial)^{-1/2} q (\kappa+\partial)^{-1/2} \Bigr\|^2_{\mathfrak I_2(\R)} \approx \int \log\bigl(4 + \tfrac{\xi^2}{\kappa^2}\bigr)\frac{|\hat q(\xi)|^2\,d\xi}{\sqrt{4\kappa^2 + \xi^2}}.
\end{equation}
Here the symbol $\approx$ indicates that the ratio of the two sides lies between two positive absolute constants.
\end{lemma}

\begin{proof}
By Plancherel,
$$
\text{LHS\eqref{AKNS I2}} = \frac{1}{2\pi} \iint \frac{|\hat q(\xi-\eta)|^2}{\sqrt{(\kappa^2+\xi^2)(\kappa^2+\eta^2)}}\,d\xi\,d\eta .
$$

On the other hand,
$$
\int \frac{dx}{\sqrt{1+(x+y)^2}\sqrt{1+(x-y)^2}} \approx \frac{ \log(4+4y^2) }{\sqrt{1+y^2}},
$$
as one readily sees by breaking the region of integration into the pieces $|x|<2|y|$ and $|x|>2|y|$.
Note that the logarithmic growth originates from the former piece.  The lemma now follows by choosing $\xi=\kappa(x+y)$ and $\eta=\kappa(x-y)$.
\end{proof}

In view of the above, the series \eqref{AKNS alpha} converges (geometrically) and can be differentiated term-by-term as soon as $\kappa$ is large enough so that
\begin{align}\label{k small}
\int \log\bigl(4 + \tfrac{\xi^2}{\kappa^2}\bigr)\frac{|\hat q(\xi)|^2\,d\xi}{\sqrt{4\kappa^2 + \xi^2}} < c,
\end{align}
for some absolute constant $c>0$.

Unlike the KdV case, our basic operator here is not self-adjoint.  Thus the norm above does not tell us the size of the leading term in the power series \eqref{AKNS alpha}.  To fill this void, we prove the following:

\begin{lemma}\label{L:AKNS tr}
For $\kappa>0$ and $q(x)$ Schwartz,
\begin{equation}\label{AKNS tr}
\Re \tr\bigl\{(\kappa-\partial)^{-1} q (\kappa+\partial)^{-1} \bar q \bigr\} = \int \frac{2\kappa |\hat q(\xi)|^2\,d\xi }{4\kappa^2 + \xi^2}.
\end{equation}
\end{lemma}

\begin{proof}
It is easy to verify that for $\kappa>0$, the operator $(\kappa-\partial)^{-1}$ admits the following integral kernel
$$
k_-(x,y) = \begin{cases} e^{\kappa(x-y)} & \text{: if $x<y$} \\ 0 & \text{: otherwise,} \end{cases}
$$
while its adjoint $(\partial+\kappa)^{-1}$ admits the kernel
$$
k_+(x,y) = \begin{cases} e^{-\kappa(x-y)} & \text{: if $x>y$} \\ 0 & \text{: otherwise.} \end{cases}
$$
From this and Proposition~\ref{P:free}, we deduce that
\begin{align*}
\text{LHS\eqref{AKNS tr}}  &= \Re \iint_{x<y} e^{2\kappa(x-y)} q(y)\bar q(x)\,dx\,dy \\
&= \tfrac12 \iint e^{-2\kappa|x-y|} q(y)\bar q(x)\,dx\,dy  = \text{RHS\eqref{AKNS tr}},
\end{align*}
which completes the proof of the lemma.
\end{proof}

The next two propositions guarantee the constancy of the perturbation determinant along the flows generated by \eqref{NLS} and \eqref{HmKdV}, respectively.

\begin{prop}[Conservation of $\alpha$ for NLS]\label{P:NLS dot} Let $q(t,x)$ denote a Schwartz-space solution to \eqref{NLS}.  Then for $\kappa$ large enough so that \eqref{k small} holds,
\begin{equation*}
\ddt \alpha(\kappa;q(t)) = 0.
\end{equation*}
\end{prop}

\begin{proof}
As in the treatment of the KdV case, it suffices to show that
\begin{align}\label{NLS head}
\tr\bigl\{  (\kappa-\partial)^{-1} q\,'' (\kappa+\partial)^{-1} \bar q - (\kappa-\partial)^{-1} q (\kappa+\partial)^{-1} \bar q\,''\bigr\}&=0
\end{align}
and that for each $\ell\geq 1$,
\begin{equation}\label{NLS tele}
\left\{\begin{aligned}
&\tr\Bigl\{ \bigl[(\kappa-\partial)^{-1} q (\kappa+\partial)^{-1} \bar q \bigr]^{\ell}  (\kappa-\partial)^{-1} q\,'' (\kappa+\partial)^{-1} \bar q \\
&\qquad  - \bigl[(\kappa-\partial)^{-1} q (\kappa+\partial)^{-1} \bar q \bigr]^{\ell}  (\kappa-\partial)^{-1} q (\kappa+\partial)^{-1} \bar q\,'' \Bigr\} \\
{}={}& \tr\Bigl\{\pm \bigl[(\kappa-\partial)^{-1} q (\kappa+\partial)^{-1} \bar q \bigr]^{\ell-1}  (\kappa-\partial)^{-1} 2 |q|^2 q (\kappa+\partial)^{-1} \bar q \\
&\qquad  \mp \bigl[(\kappa-\partial)^{-1} q (\kappa+\partial)^{-1} \bar q \bigr]^{\ell-1}  (\kappa-\partial)^{-1} q (\kappa+\partial)^{-1} 2 |q|^2  \bar q \Bigr\}.
\end{aligned}\right.
\end{equation}

To verify \eqref{NLS head} we argue as in the proof of Lemma~\ref{L:AKNS tr} and then integrate by parts:
\begin{align*}
\text{LHS\eqref{NLS head}} &= \iint_{x<y} e^{2\kappa(x-y)} [q''(y)\bar q(x)- q(y)\bar q''(x)]\,dx\,dy \\
&= \iint e^{-2\kappa|x-y|}\signum(y-x) q''(y)\bar q(x) \,dx\,dy \\
&= - \iint e^{-2\kappa|x-y|}\signum(y-x) q'(y)\bar q'(x) \,dx\,dy =0.
\end{align*}

The veracity of \eqref{NLS tele} follows readily from the elementary operator identities:
\begin{align*}
     q'' &= q (\partial^2-2\kappa\partial-\kappa^2) + (\partial^2+2\kappa\partial-\kappa^2) q  + 2 (\kappa-\partial) q (\kappa+\partial),\\
\bar q'' &= (\partial^2-2\kappa\partial-\kappa^2) \bar q + \bar q (\partial^2+2\kappa\partial-\kappa^2) + 2 (\kappa+\partial) \bar q (\kappa-\partial),
\end{align*}
valid for all $\kappa\in\R$.  (The second identity here follows from the first by complex conjugation and reversing the sign of $\kappa$.)  Indeed, the last term in each identity produces one of the terms on RHS\eqref{NLS tele}, albeit in the opposite order in which they appear there.  The net contribution of the remaining terms in each identity is zero.  More specifically, the contribution of the first term in the identity for $q''$ precisely cancels that of the first term from the identity for $\bar q''$ due to the commutativity of constant-coefficient differential operators.  Similarly, the second term from the $q''$ identity cancels its partner in the $\bar q''$ identity, after additionally cycling the trace.
\end{proof}

\begin{prop}[Conservation of $\alpha$ for Hirota mKdV]\label{P:HmKdV dot} Let $q(t,x)$ denote a Schwartz-space solution to \eqref{HmKdV}.  Then for $\kappa$ large enough so that \eqref{k small} holds,
\begin{equation*}
\ddt \alpha(\kappa;q(t)) = 0.
\end{equation*}
\end{prop}

\begin{proof}
Our arguments parallel the proof of Proposition~\ref{P:NLS dot} very closely.   As there, it suffices to show that
\begin{align}
\tr\bigl\{  (\kappa-\partial)^{-1} q\,''' (\kappa+\partial)^{-1} \bar q + (\kappa-\partial)^{-1} q (\kappa+\partial)^{-1} \bar q\,'''\bigr\}&=0 \label{HmKdV head}
\end{align}
and that for each $\ell\geq 1$,
\begin{equation}\label{HmKdV tele}
\left\{\begin{aligned}
&\tr\Bigl\{ \bigl[(\kappa-\partial)^{-1} q (\kappa+\partial)^{-1} \bar q \bigr]^{\ell}  (\kappa-\partial)^{-1} q\,''' (\kappa+\partial)^{-1} \bar q \\
&\quad\ \ + \bigl[(\kappa-\partial)^{-1} q (\kappa+\partial)^{-1} \bar q \bigr]^{\ell}  (\kappa-\partial)^{-1} q (\kappa+\partial)^{-1} \bar q\,''' \Bigr\} \\
{}={}& \tr\Bigl\{-\bigl[(\kappa-\partial)^{-1} q (\kappa+\partial)^{-1} \bar q \bigr]^{\ell-1}  (\kappa-\partial)^{-1} 6 |q|^2 q' (\kappa+\partial)^{-1} \bar q \\
&\qquad  - \bigl[(\kappa-\partial)^{-1} q (\kappa+\partial)^{-1} \bar q \bigr]^{\ell-1}  (\kappa-\partial)^{-1} q (\kappa+\partial)^{-1} 6 |q|^2  \bar q' \Bigr\}.
\end{aligned}\right.
\end{equation}

Using the integral kernels introduced in the proof of Lemma~\ref{L:AKNS tr} and integrating by parts repeatedly shows
\begin{align*}
\text{LHS\eqref{HmKdV head}} &= \iint_{x<y} e^{2\kappa(x-y)} [q'''(y)\bar q(x) + q(y)\bar q'''(x)]\,dx\,dy \\
&= \int_\R [q(z)\bar q''(z)- q''(z)\bar q(z)] \,dz \\
&\qquad +  2\kappa \iint_{x<y} e^{2\kappa(x-y)} [q''(y)\bar q(x) - q(y)\bar q''(x)]\,dx\,dy \\
&= 0 - 2\kappa  \int_\R [q'(z)\bar q(z) + q(z)\bar q'(z)] \,dz \\
&\qquad +  4\kappa^2 \iint_{x<y} e^{2\kappa(x-y)} [q'(y)\bar q(x) + q(y)\bar q'(x)]\,dx\,dy \\
&= 0 - 4\kappa^2  \int_\R [q(z)\bar q(z) - q(z)\bar q(z)] \,dz \\
&\qquad +  8\kappa^3 \iint_{x<y} e^{2\kappa(x-y)} [q(y)\bar q(x) - q(y)\bar q(x)]\,dx\,dy  \\
&=0.
\end{align*}
Note that the $z$ integrals above arise by combining the boundary terms that appear when integrating by parts either in $x$ or $y$.  They themselves vanish because they are integrals of complete derivatives.

We now turn to \eqref{HmKdV tele}.  As a first step, we employ the identities:
\begin{align*}
q''' &= - q (\partial^3-3\kappa\partial^2+3\kappa^2\partial+3\kappa^3) + (\partial^3+3\kappa\partial^2+3\kappa^2\partial-3\kappa^3) q  \\
&\qquad      + (\kappa-\partial) [3 q'+6\kappa q] (\kappa+\partial),\\
\bar q''' &= + (\partial^3-3\kappa\partial^2+3\kappa^2\partial+3\kappa^3) \bar q - \bar q (\partial^3+3\kappa\partial^2+3\kappa^2\partial-3\kappa^3)   \\
&\qquad      + (\kappa+\partial) [3 \bar q'-6\kappa \bar q] (\kappa-\partial).
\end{align*}
(As previously, the second identity here follows from the first by complex conjugation and reversing the sign of $\kappa$.)  The contribution of the first two terms from each identity cancel one another.  In this way, we see that
\begin{align*}
\text{LHS\eqref{HmKdV tele}} =
&\tr\Bigl\{ \bigl[(\kappa-\partial)^{-1} q (\kappa+\partial)^{-1} \bar q \bigr]^{\ell-1}  (\kappa-\partial)^{-1} q (\kappa+\partial)^{-1} \bar q^2[3q'+6\kappa q] \\
&\quad \ \ + \bigl[(\kappa-\partial)^{-1} q (\kappa+\partial)^{-1} \bar q \bigr]^{\ell-1}  (\kappa-\partial)^{-1} q^2[3\bar q'-6\kappa \bar q] (\kappa+\partial)^{-1} \bar q \Bigr\} .
\end{align*}
We will now reduce RHS\eqref{HmKdV tele} to the same form.  This relies on the following:
\begin{align*}
6 |q|^2 q' &= -  3 |q|^2 q (\kappa+\partial) - 3 (\kappa-\partial) |q|^2 q  - q^2[3\bar q' - 6\kappa \bar q], \\
6 |q|^2 \bar q' &= + 3 (\kappa+\partial) |q|^2 \bar q +  3 |q|^2 \bar q (\kappa-\partial) - \bar q^2[3q' + 6\kappa q].
\end{align*}
Plugging these identities into RHS\eqref{HmKdV tele}, we see that the net contribution of the two first terms is zero and likewise, that of the pair of second terms.  The final terms in each identity produce each of the two terms appearing in the formula for LHS\eqref{HmKdV tele} above, albeit in reversed order.  Thus
$$
\text{LHS\eqref{HmKdV tele}} = \text{RHS\eqref{HmKdV tele}}
$$
and the proposition is proved.
\end{proof}

\begin{theorem}\label{T:AKNS Besov}
Fix $-\tfrac12 < s <0$ and $1\leq r\leq \infty$ and let $q(t)$ be a Schwartz solution to \eqref{NLS} or \eqref{HmKdV}.  Then
\begin{align*}
\| q(t) \|_{B^{s,2}_r} \lesssim  \| q(0) \|_{B^{s,2}_r} \Bigl( 1 + \| q(0) \|_{B^{s,2}_r}^2 \Bigr)^{\frac{|s|}{1-2|s|}}.
\end{align*}
\end{theorem}

\begin{proof}
It will be convenient to introduce a norm equivalent to the Besov norm, but adapted to a frequency scale $\kappa_0\in\NN$, namely,
\begin{align*}
\| f \|_{Z_{\kappa_0}} := \biggl[ \sum_{N\in\NN} N^{rs} \biggl(\int \frac{2\kappa_0^2N^2 |\hat f(\xi)|^2\,d\xi }{4\kappa_0^2N^2 + \xi^2}\biggr)^{r/2} \,\biggr]^{1/r}.
\end{align*}
That this is an equivalent norm follows from Lemma~\ref{L:Besov1}, which shows
\begin{align}\label{AKNS Bes}
 \|f\|_{B^{s,2}_r}\lesssim \| f \|_{Z_{\kappa_0}} \lesssim \kappa_0^{|s|} \|f\|_{B^{s,2}_r}.
\end{align}

From \eqref{AKNS I2} one easily sees that for $\kappa\geq\kappa_0$,
\begin{align}
\Bigl\| (\kappa-{}&\partial )^{-1/2} q (\kappa+\partial)^{-1/2} \Bigr\|^2_{\mathfrak I_2(\R)} \notag\\
&\lesssim \tfrac1{\kappa}\int_{|\xi|\leq \kappa_0} |\hat q(\xi)|^2\,d\xi
	+ \sum_{N\in\NN} \tfrac{\log(2+\kappa_0^2N^2\kappa^{-2})}{\kappa_0 N+\kappa} \int_{\kappa_0 N \leq |\xi|\leq 2 \kappa_0 N} |\hat q(\xi)|^2\,d\xi \notag\\
&\lesssim   \bigl(\tfrac{\kappa_0}{\kappa}\bigr)^{1-2|s|} \kappa_0^{-1} \| q \|_{Z_{\kappa_0}}^2 \label{11:25} \\
&\lesssim  \kappa^{2|s|-1} \| q \|_{B^{s,2}_r}^2 \label{11:26}.
\end{align}
(The penultimate step is simplest when $r=\infty$, which then implies the other cases.)  Correspondingly, for
$$
\kappa_0 \geq C  \bigl(1+ \| q(0) \|_{B^{s,2}_r}^2\bigr)^{\frac{1}{1-2|s|}}
$$
with some large absolute constant $C$, we have \eqref{k small} for $t\in I$, a small neigborhood of the temporal origin.  On this time interval, we then obtain
\begin{align*}
\biggl| \alpha(\kappa,q(t)) - \int \frac{2\kappa |\hat q(t,\xi)|^2\,d\xi }{4\kappa^2 + \xi^2} \biggr| \lesssim \kappa_0^{-4|s|} \kappa^{4|s|-2} \| q(t) \|_{Z_{\kappa_0}}^4
\end{align*}
and so, by the conservation of $\alpha$, we then deduce that for $t\in I$,
\begin{align*}
\kappa_0^{-\frac12} \| q(t) \|_{Z_{\kappa_0}} = \kappa_0^{-\frac12} \| q(0) \|_{Z_{\kappa_0}}
	+ O\biggl( \Bigl[ \kappa_0^{-\frac12} \| q(0) \|_{Z_{\kappa_0}} \Bigr]^2 + \Bigl[ \kappa_0^{-\frac12} \| q(t) \|_{Z_{\kappa_0}} \Bigr]^2 \biggr).
\end{align*}
Noting that for $\kappa_0$ as above we have
$$
\kappa_0^{-\frac12} \| q(0) \|_{Z_{\kappa_0}} \lesssim C^{|s| - \frac12},
$$
and choosing $C$ larger (if necessary), a simple continuity argument then yields
$$
\| q(t) \|_{Z_{\kappa_0}}^2 \lesssim  \| q(0) \|_{Z_{\kappa_0}}^2
$$
uniformly in time.  (Note that \eqref{11:25} shows that \eqref{k small} propagates.)  The final result then follows from \eqref{AKNS Bes} and our choice of $\kappa_0$.
\end{proof}

The methods presented in the previous section allow one to extend the result to Besov norms with $-\frac12<s<1$; however, we will not pursue this further.  Rather, we wish to discuss to what extent we can approach more closely to the forbidden end-point $s=-\tfrac12$.  To this end, we introduce the following modified Besov norms, depending on a parameter $\kappa_0>0$:
\begin{align*}
\| q \|_{Y_{\kappa_0}} :=   \max\Bigl\{ & \kappa_0^{-\frac12}\| \hat q(\xi) \|_{L^2(|\xi|\leq \kappa_0)}, \\
    & \sup_{N\in\NN}  (\kappa_0 N)^{-\frac12}\log^{2}(2N) \| \hat q(\xi) \|_{L^2(N\kappa_0 < |\xi|\leq 2N\kappa_0)} \Bigr\},
\end{align*}
which mimics $B^{-1/2,2}_\infty(\R)$ and, by analogy with $B^{-1/2,2}_2(\R)\cong H^{-1/2}(\R)$,
\begin{align*}
\| q \|_{X_{\kappa_0}}^2 :=   \kappa_0^{-1}\| \hat q(\xi) \|_{L^2(|\xi|\leq \kappa_0)}^2 + \sum_{N\in\NN}  \frac{\log^{3}(2N)}{\kappa_0 N} \| \hat q(\xi) \|_{L^2(N\kappa_0 < |\xi|\leq 2N\kappa_0)}^2.
\end{align*}
It is not difficult to verify that due to the particular powers of the logarithm involved, neither of these spaces includes the other.

As a rather trivial consequence of~\eqref{AKNS tr}, we can link these norms to the leading term in the series \eqref{AKNS alpha} defining the perturbation determinant:

\begin{lemma}\label{L:XY}
Uniformly in $\kappa_0>0$, we have the following equivalences
\begin{align*}
  \| q \|_{Y_{\kappa_0}}^2  &\approx \sup_{M\in\NN} \log^{4}(2M) \Re \tr\bigl\{(\kappa_0M-\partial)^{-1} q (\kappa_0M+\partial)^{-1} \bar q \bigr\}, \\
  \| q \|_{X_{\kappa_0}}^2  &\approx \sum_{M\in\NN} \log^{3}(2M) \Re \tr\bigl\{(\kappa_0M-\partial)^{-1} q (\kappa_0M+\partial)^{-1} \bar q \bigr\}.
\end{align*}
\end{lemma}

As we shall see, this leads rather quickly to our final result, which essentially says that these new norms are conserved by the flow.

\begin{theorem}\label{T:XY}
Let $q(t)$ be a Schwartz solution to \eqref{NLS} or \eqref{HmKdV}. If $\kappa_0$ is large enough so that
\begin{align*}
\| q(0) \|_{Y_{\kappa_0}} \leq c, \qtq{then} C^{-1} \| q(0) \|_{Y_{\kappa_0}} \leq \sup_{t\in\R}\| q(t) \|_{Y_{\kappa_0}} \leq C \| q(0) \|_{Y_{\kappa_0}}
\end{align*}
for some absolute constants $c,C>0$.  Analogously,
\begin{align*}
\| q(0) \|_{X_{\kappa_0}} \leq c \quad\implies\quad C^{-1} \| q(0) \|_{X_{\kappa_0}} \leq\sup_{t\in\R}\| q(t) \|_{X_{\kappa_0}} \leq C \| q(0) \|_{X_{\kappa_0}}.
\end{align*}
\end{theorem}

\begin{proof}
The main part of the proof is to obtain a suitable estimate on the tail of the series \eqref{AKNS alpha}.  In view of Lemma~\ref{L:AKNS I2}, we have
\begin{align*}
\log^{2}(2M) & \Bigl\|(\kappa_0M -\partial)^{-1/2} q (\kappa_0M+\partial)^{-1/2} \Bigr\|^2_{\mathfrak I_2(\R)}  \\
&\lesssim \frac{\log^{2}(2M)}{\kappa_0M} \int_{|\xi|\leq \kappa_0} |\hat q(\xi)|^2\,d\xi \\
&\qquad    +  \sum_{N\in\NN} \frac{\log^{2}(2M)\log\bigl(2+\tfrac NM\bigr)}{\kappa_0M+\kappa_0N} \int_{\kappa_0N < |\xi|\leq 2\kappa_0N} |\hat q(\xi)|^2\,d\xi \\
&\lesssim  \| q \|_{Y_{\kappa_0}}^2 \biggl[ \frac{\log^2(2M)}{M}  +  \sum_{N\in\NN} \frac{\log^{2}(2M)\log\bigl(2+\tfrac NM\bigr) N}{(M+N)\log^{4}(2N)} \biggr] \\
&\lesssim \| q \|_{Y_{\kappa_0}}^2
\end{align*}
uniformly for $M\in\NN$.  Analogously,
\begin{align*}
\sum_{M\in\NN} & \log^{3}(2M)  \Bigl\|(\kappa_0M -\partial)^{-1/2} q (\kappa_0M+\partial)^{-1/2} \Bigr\|^4_{\mathfrak I_2(\R)}  \\
&\lesssim \sum_{M\in\NN}  \biggl[ \frac{\log^{3/2}(2M)}{\kappa_0M} \int_{|\xi|\leq \kappa_0} |\hat q(\xi)|^2\,d\xi \biggr]^2 \\
&\ \ \ \   + \sum_{M\in\NN}  \biggl[  \sum_{N\in\NN} \frac{\log^{3/2}(2M)\log\bigl(2+\tfrac NM\bigr)}{\kappa_0M+\kappa_0N} \int_{\kappa_0N < |\xi|\leq 2\kappa_0N} |\hat q(\xi)|^2\,d\xi \biggr]^2\\
&\lesssim \| q \|_{X_{\kappa_0}}^4  \biggl[ 1+ \sup_{N\in\NN}\sum_{M\in\NN}  \frac{\log^{3}(2M)\log^2\bigl(2+\tfrac NM\bigr)N^2}{(M+N)^2\log^6(2N)} \biggr]
    \lesssim \| q \|_{X_{\kappa_0}}^4.
\end{align*}

Combining the above with Lemma~\ref{L:XY} and conservation of $\alpha(\kappa;q(t))$ we see that
\begin{align*}
\sup_{|t|\leq T} \| q(t) \|_{Y_{\kappa_0}}^2 \lesssim \| q(0) \|_{Y_{\kappa_0}}^2 + \sup_{|t|\leq T} \| q(t) \|_{Y_{\kappa_0}}^4,
\end{align*}
provided the left-hand side is sufficiently small (to guarantee convergence of the series \eqref{AKNS alpha}).  The result then follows by a simple continuity argument.
\end{proof}

\end{document}